\documentclass[leqno]{amsart}
\usepackage{caption}
\usepackage{amsmath,amssymb,amsthm,mathtools,mathrsfs,tikz}
\usepackage{environ}
\usepackage{stmaryrd}
\usetikzlibrary{arrows}
\usetikzlibrary{positioning}
\usetikzlibrary{decorations.text}
\usetikzlibrary{decorations.pathmorphing}
\usetikzlibrary{decorations.pathreplacing}
\makeatletter
\newsavebox{\@brx}
\newcommand{\llangle}[1][]{\savebox{\@brx}{\(\m@th{#1\langle}\)}%
	\mathopen{\copy\@brx\kern-0.5\wd\@brx\usebox{\@brx}}}
\newcommand{\rrangle}[1][]{\savebox{\@brx}{\(\m@th{#1\rangle}\)}%
	\mathclose{\copy\@brx\kern-0.5\wd\@brx\usebox{\@brx}}}
\makeatother
\makeatletter
\newsavebox{\measure@tikzpicture}
\NewEnviron{scaletikzpicturetowidth}[1]{%
	\mathrm{Def}\tikz@width{#1}%
	\mathrm{Def}\tikzscale{1}\begin{lrbox}{\measure@tikzpicture}%
		\BODY
	\end{lrbox}%
	\pgfmathparse{#1/\wd\measure@tikzpicture}%
	\edef\tikzscale{\pgfmathresult}%
	\BODY
}
\makeatother
\usepackage{bbm}
\DeclarePairedDelimiter\norm{\lvert}{\rvert}

\usepackage{bm}
\usepackage[colorlinks=true,linkcolor=blue,citecolor=blue]{hyperref}
\usepackage{enumitem}
\usepackage{csquotes}

\def\irr#1{{\rm  Irr}(#1)}
\def\cd#1{{\rm  cd}(#1)}

\numberwithin{equation}{section}
\newcounter{intro}

\newtheorem{introthm}[intro]{Theorem}

\newtheorem{thm}[equation]{Theorem}
\newtheorem{lem}[equation]{Lemma}
\newtheorem{cor}[equation]{Corollary}

\theoremstyle{remark}

\theoremstyle{definition}

\title{GVZ-groups, Flat groups, and CM-Groups}

\author{Shawn T. Burkett}
\address{Department of Mathematical Sciences, Kent State University, Kent,
	Ohio 44242, U.S.A.} \email{sburket1@kent.edu}

\author{Mark L. Lewis}
\address{Department of Mathematical Sciences, Kent State University, Kent,
	Ohio 44242, U.S.A.} \email{lewis@math.kent.edu}

\date{\today}
\keywords{GVZ groups; $p$-groups; fully-ramified characters; flat groups}
\subjclass[2010]{20C15}

\begin{document}

\begin{abstract}
We show that a group is a GVZ-group if and only if it is a flat group.  We show that the nilpotence class of a GVZ-group is bounded by the number of distinct degrees of irreducible characters.  We also show that certain CM-groups can be characterized as GVZ-groups whose irreducible character values lie in the prime field.
\end{abstract}

\maketitle

\section{Introduction}
In this note, all groups are finite, and we will write $\irr G$ for the set of irreducible characters of a group $G$.  It is well-known that the inequality $\chi(1) \le \norm {G:Z(\chi)}^{1/2}$ holds for all irreducible characters $\chi \in \irr{G}$, where $Z (\chi) = \{ g \in G \mid \norm{\chi(g)} = \chi(1) \}$ is the center of $\chi$.  Furthermore, equality holds if and only if every element of $G$ satisfies one of the two conditions: $g \in Z(\chi)$ or $\chi (g) = 0$.  (See Chapter 2 of \cite{MI76}.)

It has been observed that there is a parallelism between the irreducible characters of a group and the conjugacy classes of group.  Furthermore, there seems to be a mysterious relationship between the degrees of the irreducible characters and the sizes of the conjugacy classes. Chillag has a nice exposition about these parallels in \cite{chillag}.  With this in mind, it is not difficult to see that the inequality $\norm{\mathrm{cl}_G(g)}\le\norm{[g,G]}$ holds for every element $g\in G$, where $[g,G]$ is the (normal) subgroup of $G$ generated by the set $\gamma_G (g) = \{ [x,g] \mid x \in G\}$ of commutators involving $g$.  Furthermore, equality holds if and only if $\gamma_G (g)$ is a subgroup of $G$.  (See Corollary \ref{basics 1}.)

In this paper, we consider groups where these extremes occur.  Following \cite{AN12gvz}, we say $G$ is a {\it GVZ-group} if $\chi (1) = \norm {G:Z(\chi)}^{1/2}$ for every irreducible character $\chi \in \irr G$.   A group $G$ is a {\it flat group}, following \cite{flatness}, if $\norm{\mathrm{cl}_G(g)}=\norm{[g,G]}$ holds for every element $g\in G$. Interestingly, these two extreme situations can only occur simultaneously.  

\begin{introthm}\label{flat = gvz}
The group $G$ is flat if and only if it is a GVZ-group.
\end{introthm}

Let $G$ be a GVZ-group. Then $G$ is necessarily nilpotent (see \cite[Proposition 3.2]{OnoType} or \cite[Proposition 1.2]{AN12gvz}). Since nilpotent groups are $M$-groups, Taketa's theorem implies that the derived length of $G$ is bounded by the number distinct degrees of the irreducible characters of $G$.  That is, if ${\rm dl} (G)$ is the derived length of $G$ and $\cd G = \{ \chi (1) \mid \chi \in \irr G\}$ is the set of irreducible character degrees, then ${\rm dl} (G) \le |\cd G|$.  

In general, there is not a bound between the number of character degrees and the nilpotence class of a nilpotent group.  In particular, all dihedral, semi-dihedral and generalized quaternion $2$-groups have the character degree set $\{ 1, 2\}$ and there exist examples of each of these groups with arbitrarily large nilpotence classes.  In fact, there has been some research on which sets of character degrees bound the nilpotence class of a $p$-group.  (See \cite{class1}, \cite{class2}, and \cite{class3}.)

We adapt the classical Taketa argument to show that the nilpotence class of a GVZ-group is bounded by the number of irreducible character degrees.  If $G$ is a nilpotent group, then we write $c(G)$ for the nilpotence class of $G$.

\begin{introthm}\label{taketa}
If $G$ is a GVZ-group, then $c (G) \le |\cd G|$.
\end{introthm}

We mention that Theorem \ref{taketa} is not without content, as there exist GVZ-groups of arbitrarily high nilpotence class. Such examples are constructed in \cite{ ML19gvz, AN16gvz} and will be discussed in Section~\ref{gvz section}.  We also note that Professor Mann has pointed out that Golikova and Starostin had previously constructed GVZ-groups that are $2$-groups of arbitrarily high nilpotence class in \cite{GoSt}.

As a separate application, we illustrate a connection between GVZ-groups and another type of group called a ${\rm CM}_n$-group. Initially studied in \cite{BZ}, a group $G$ is called a ${\rm CM}_n$-group if for every normal subgroup of $G$ appears as the kernel of at most $n$ irreducible characters of $G$. Of particular interest in this note are the ${\rm CM}_{p-1}$-groups where $p$ is a prime. Specifically, we show the following.  

\begin{introthm} \label{CMp-1-intro}
Let $p$ be a prime and let $G$ be a $p$-group.  Then $G$ is a ${\rm CM}_{p-1}$-group if and only if $G$ is a GVZ-group and every character in $\irr G$ has values in the $p^{\rm th}$ cyclotomic field.
\end{introthm}

A ${\rm CM}_1$-group as defined above (not to be confused with the definition given in \cite{saeidi}) is often called a ${\rm CM}$-group. As a consequence of Theorem \ref{CMp-1}, we see that a $2$-group is a ${\rm CM}$-group if and only if it is a GVZ-group and a rational group. 

We close this introduction by thanking Professor Mann and Professor Abdollahi for comments regarding the earlier paper ``GVZ-groups'' that has been subsumed by this current paper.

\section{Conjugacy classes of GVZ-groups}\label{gvz section}

As far as we can determine, GVZ-groups were initially studied in \cite{OnoType} under the guise of {\it groups of Ono type}.  In that paper, it was proved that GVZ-groups are nilpotent.   Motivated by problem of Berkovich suggested in \cite{YBGPPOV1}, Nenciu introduced the definition of GVZ-groups in \cite{AN12gvz}.  In \cite{AN12gvz, AN16gvz}, Nenciu studies GVZ-groups that have the additional condition where the centers of the irreducible characters form a chain of subgroups.  Further properties of GVZ-groups can be found in \cite{SBMLnestedclass, ML19gvz}.  

A flat group is one where the equality $\mathrm{cl}_G(g)=g[g,G]$ holds for every $g\in G$. Finite flat groups were initially studied in \cite{flatness}, where a flat group is defined to be one in which every conjugacy class is a coset of a (necessarily normal) subgroup $N$. In Theorem 4.2 of \cite{flatness}, they prove that a (finite) flat nilpotent group is a GVZ-group. Using a different method, we prove that all flat groups are GVZ-groups. Since GVZ-groups are nilpotent, we may conclude that (finite) flat groups are in fact always nilpotent. To relate flat groups to GVZ-groups, we begin with some basic results about the subgroups $[g,G]$ and their connection to centers of characters. 

Our first lemma is an easy proof of the well-known fact that $[g,G]$ is normal in $G$.  Recall from the Introduction that the set $\gamma_G(g)$ for $g\in G$ is defined to be the set of commutators $\{[g,x]\mid x\in G\}$. We note that $\gamma_G (g)$ need not be a normal subset of $G$.  To see this consider $\gamma_{S_4} (1,2)$ and $\gamma_{S_4} (1,3)$.   This is unusual, since usually when we show that a subgroup generated by a subset is normal, we show that the subset is normal. 

\begin{lem} \label{com sub normal}
Let $G$ be a group.  If $g\in G$, then $[g,G]$ is normal in $G$. 
\end{lem}

\begin{proof}
Consider elements $x, y \in G$.  It suffices to show that $[g,x]^y \in [g,G]$.  Since $[g,xy] = [g,y] [g,x]^y$, we have $[g,x]^y = [g,y]^{-1} [g,xy]\in[g,G]$.
\end{proof}

We will also need the following result that yields a useful description of the set of irreducible characters containing $[g,G]$ in their kernel.

\begin{lem} \label{cen cond}
Let $G$ be a group.  Fix an element $g \in G$ and a character $\chi \in \irr G$.  Then $g \in Z(\chi)$ if and only if $[g,G] \le \ker (\chi)$.
\end{lem}

\begin{proof}
This follows immediately from the definition that states: $Z(\chi)/\ker (\chi) = Z (G/\ker (\chi))$.
\end{proof}

We next present a lemma that is motivated by Lemma 1 of \cite{camina}, Proposition 3.1 of \cite{ChMc}, Lemma 2.1 of \cite{gencam}, and Lemma 2.1 of \cite{mlaiki}.  We also refer the reader to Lemmas 2.1 and 2.2 of the first author's expository paper \cite{mycam}. We also note that  (1) implies (4)  is Lemma 3.1 of \cite{squaringclasses}.  Also, using Theorem 2.2 of \cite{squaringclasses}, one can derive that implication of (4) implies (1) when $|G|$ is odd.

\begin{lem} \label{basics}
Let $M$ be a normal subgroup of $G$ and let $g \in G \setminus M$.  Then the following are equivalent:
\begin{enumerate}[label={\bf(\arabic*)}]
\item $g$ is conjugate to every element in $gM$.
\item For every element $z \in M$, there exists an element $x \in G$ so that $[g,x] = z$.
\item $|C_G (g)| = |C_{G/M} (gM)|$.
\item $\chi (g) = 0$ for every character $\chi \in \irr {G \mid M}$.
\end{enumerate}
\end{lem}

\begin{proof}
We first show (1) and (2) are equivalent.  Notice that if $z \in M$, then $g$ and $gz$ are conjugate if and only if there exists $x \in G$ so that $g^x = gz$.  However, we see that $g^x = gz$ if and only if $[g,x] = g^{-1}g^x = g^{-1}gz = z$.  Hence, $g$ will be conjugate to every element in $gM$ if and only if for every element $z \in M$, there is an element $x \in G$ so that $[g,x] = z$.
	
We next show that (1) and (3) are equivalent.  Note that
$$
{\rm cl} (g) \subseteq \bigcup_{x \in G} (gM)^x.
$$
We know that $|G:C_G (g)| = |{\rm cl}_G(g)|$.  On the other hand, $|\bigcup_{x \in G} (gM)^x|$ will equal the number of conjugates of $gM$ times the size of $M$.  Thus, we have the equality $|\bigcup_{x \in G} (gM)^x| = |G/M:C_{G/M} (gM)| |M|$.  It follows that ${\rm cl}_G(g) = \bigcup_{x \in G} (gM)^x$ if and only if $|C_G (g)| = |C_{G/M} (gM)|$.  On the other hand, ${\rm cl}_G(g) = \bigcup_{x \in G} (gM)^x$ if and only if $g$ is conjugate to all elements in $gM$.  This implies that (1) and (3) are equivalent.
	
Now, we show (3) and (4) are equivalent.  By the second Orthogonality relation, we have
$$
|C_G (g)| = \sum_{\chi \in \irr G} |\chi (g)|^2 = \sum_{\chi \in \irr {G/M}} |\chi (g)|^2 + \sum_{\chi \in \irr {G \mid M}} |\chi (g)|^2,
$$
and
$$
|C_{G/M} (gM)| = \sum_{\chi \in \irr {G/M}} |\chi (g)|^2.
$$
This implies that we have $\sum_{\chi \in \irr {G \mid M}} |\chi (g)|^2 = 0$ if and only if the equality $|C_G (g)| = |C_{G/M} (gM)|$ holds.  Since $|\chi (g)|^2$ is a nonnegative real number for all $\chi \in \irr {G \mid M}$, we conclude that $\sum_{\chi \in \irr {G \mid M}} |\chi (g)|^2 = 0$ if and only if $\chi (g) = 0$ for all $\chi \in \irr {G \mid M}$.  This shows that (3) and (4) are equivalent.
\end{proof}

Applying Lemma~\ref{basics} to the normal subgroup $[g,G]$ for an element $g\in G$, we obtain the following Corollary. We remark that this result appears to have been known in \cite{OnoType}, although no proof is given there.  

\begin{cor} \label{basics 1}
Let $G$ be a group and fix an element $g \in G \setminus Z(G)$.  Then the following are equivalent:
\begin{enumerate}[label={\bf(\arabic*)}]
\item $\gamma_G (g)$ is a subgroup of $G$. 
\item $\mathrm{cl}_G (g) = g[g,G]$.
\item $\chi (g) = 0$ for all characters $\chi \in \irr {G}$ satisfying $g\notin Z(\chi)$.
\end{enumerate}
\end{cor}

\begin{proof}
Observe that $\mathrm{cl}_G(g)=g\gamma_G(g)\subseteq g[g,G]$, and thus $g$ is flat if and only if $g$ is conjugate to every element of $[g,G]$, which happens if and only if $\gamma_G(g)$ is a subgroup (i.e. coincides with $[g,G]$). By Lemma~\ref{cen cond}, the set of all irreducible characters $\chi$ satisfying $g\notin Z(\chi)$ is exactly $\irr{G\mid[g,G]}$. The result thus follows easily by Lemma~\ref{basics}.
\end{proof}

We may now easily deduce Theorem~\ref{flat = gvz}.

\begin{proof}[Proof of Theorem~\ref{flat = gvz}] Apply the result of Corollary~\ref{basics 1} to every element $g\in G$.
\end{proof}

As Theorem~\ref{flat = gvz} illustrates, GVZ-groups can be characterized in terms of a condition on its conjugacy classes. We mention one more characterization of GVZ-groups in terms of conjugacy classes, although this particular characterization applies only to GVZ-groups of odd order. We thank the referee of an earlier paper entitled ``GVZ-groups'' that has been subsumed by this current paper for suggesting this next result, which follows immediately from a result of Guralnick and Navarro appearing in \cite{squaringclasses}.

\begin{thm}
Let $G$ be a group of odd order. Then $G$ is a GVZ-group if and only if $\mathrm{cl}_G(g)^2$ is a conjugacy class for every element $g\in G$.
\end{thm}

\begin{proof}
Let $g\in G$. It is not difficult to see that $C_G(g)\le C_G(g^2)$. Since $G$ has odd order, $\norm{\mathrm{cl}_G(g^2)}=\norm{\mathrm{cl}_G(g)}$. Thus the result follows immediately from \cite[Theorem 2.2]{squaringclasses}.
\end{proof}

\section{A Taketa analog}

In this section, we give a proof of Theorem~\ref{taketa}. Before doing so, we discuss the derived length and nilpotence class of GVZ-groups. In Section 5 of \cite{AN16gvz}, Nenciu constructs GVZ $p$-groups of arbitrarily large nilpotence class with fixed exponent $p$. In Example 3 in Section 8 of \cite{ML19gvz}, Lewis gives a different construction, which illustrates the existence of GVZ-groups of arbitrarily high nilpotence class and exponent. One may verify that each of these examples has derived length two. It is unknown if there exist GVZ-groups of arbitrarily high derived length. On page 3218 of \cite{YB00}, Berkovich states (with no supporting explanation) that the derived length of a GVZ-group is probably bounded.  We will also see in Section \ref{CMgroups} that it has been conjectured in the special case studied in that section that those groups are metabelian.

We now show that the nilpotence class of a GVZ-group can be bounded in terms of the number of distinct degrees of its irreducible characters, denoted $\mathrm{cd}(G)$. For convenience, we set some notation.  We let $G_i$ denote the $i^{\rm th}$ member of the {\it lower central series}. That is, we set $G_1 = G$ and we inductively define $G_{i+1} = [G_i,G]$ for every integer $i \ge 1$.  For a nilpotent group $G$, recall that $c(G)$ is the nilpotence class of $G$ --- the smallest integer for which $G_{c(G)+1}=1$.  The reader should compare this proof with the proof of Taketa's theorem (see \cite[Theorem5.12]{MI76}).  

\begin{proof}[Proof of Theorem~\ref{taketa}]
Let $1 = d_1 < d_2 < \dotsb < d_n$ be the distinct degrees in $\cd G$.  We work by induction on $\norm{G}$. If $G$ is abelian, then $c(G) = 1 = \norm{\cd G}$, and the result holds.  Thus, we may assume that $G$ is nonabelian.  Consider a character $\chi \in \irr G$.  If $\ker(\chi) > 1$, then $\norm{G/\ker(\chi)} < \norm {G}$ and $\cd {G/\ker (\chi)} \subseteq \cd G$.  By the inductive hypothesis, we have that $G_n \le \ker (\chi)$. Thus, if $G$ does not have a faithful irreducible character, then $G_n \le \bigcap_{\chi \in \irr G} \ker (\chi) = 1$. Therefore, we may assume that there exists a character $\chi \in \irr G$ with $\ker (\chi) = 1$. This implies that $Z (\chi) = Z (G)$. We have $d_i^2 \le \norm{G:Z(G)} = \chi(1)^2$ for ever integer $i$ with $1 \le i \le n$, and so, $\chi (1) = d_n$. Notice that if $a \in \cd {G/Z(G)}$, then $a^2 < \norm {G:Z(G)} = d_n^2$. It follows that $\norm{\cd {G/Z(G)}} \le n-1$. By the inductive hypothesis, this implies $G_{n-1}\le Z(G)$.  We conclude that $G_n=1$, as desired. 
\end{proof}

\section{${\rm CM}_n$-groups and GVZ-groups}\label{CMgroups}

As we stated in the Introduction, ${\rm CM}_n$-groups were defined in \cite{BZ}.  So far as we can tell, ${\rm CM}$-groups were initially studied in \cite{zhmud}.  The reader may also want to consult \cite{zhmud1} for further results regarding these groups.  We begin with a lemma that gives a lower bound on the number of faithful characters of a group.

\begin{lem} \label{Euler}
Let $G$ be a $p$-group and suppose that $Z(G)$ is cyclic, then the number of faithful characters in $\irr G$ is at least $\phi (|Z(G)|)$ where $\phi$ is the Euler $\phi$-function.
\end{lem}

\begin{proof}
Note that for each faithful character $\lambda \in \irr {Z(G)}$, we see that $\lambda^G$ has at least one irreducible constituent $\chi$.  Note that $\ker(\chi) \cap Z(G) = \ker \lambda = 1$, and so $\ker(\chi) = 1$.  Hence, $\chi$ is a faithful character.  On the other hand, any character in $\irr G$ will have a character of $Z(G)$ as its unique irreducible constituent when restricted to $Z(G)$.  This implies that the number of faithful characters of $G$ is at least the number of faithful irreducible characters of $Z(G)$.  Since $Z(G)$ is cyclic, the number of faithful irreducible characters equals $\phi (|Z(G)|)$.
\end{proof}

Recall that $G$ is a ${\rm CM}_n$-group if for every normal subgroup $N$, there are at most $n$ characters in $\irr G$ that have $N$ as their kernel.  Note for every group $G$, there is a minimal positive integer $n$ so that $G$ is a ${\rm  CM}_n$-group.

\begin{lem} \label{p-1}
If $p$ is a prime, $G$ is a $p$-group, and $G$ is a ${\rm CM}_n$-group, then $n \ge p-1$.
\end{lem}

\begin{proof}
If $N = \ker(\chi)$ for some $\chi \in \irr G$, then $Z(\chi)/N = Z(G/N)$ is cyclic.  By Lemma \ref{Euler}, we know that $G/N$ has at least $\phi (|Z(\chi)/N|)$ faithful irreducible characters.  Since $Z(\chi)/N$ is a $p$-group, we know $\phi (|Z(\chi)/N|) \ge p-1$.
\end{proof}

If $p$ is a prime, then we write $\mathbb{Q}_p$ for the field obtained by adjoining a $p$th root of unity to the rationals.  We now characterize the $p$-groups that are ${\rm CM}_{p-1}$-groups in the following theorem which includes Theorem \ref{CMp-1-intro} from the Introduction.  We note that the equivalence of (1) and (2) is essentially proved in Theorem 9.3.16 of \cite{BZ}, but our proof seems to be considerably shorter.  

\begin{thm} \label{CMp-1}
Let $p$ be a prime and let $G$ be a $p$-group.  Then the following are equivalent:
\begin{enumerate}[label={\normalfont\bf(\arabic*)}]
\item $G$ is a ${\rm CM}_{p-1}$-group.
\item $G$ is a GVZ-group and $|Z(\chi)/\ker(\chi)| = p$ for all $1_G \ne \chi \in \irr G$.
\item $G$ is a GVZ-group and every character in $\irr G$ has values in $\mathbb{Q}_p$.
\end{enumerate}
\end{thm}

\begin{proof}
Suppose that $G$ is a ${\rm CM}_{p-1}$-group.  Consider a character $1_G \ne \chi \in \irr G$.  By Lemma \ref{Euler}, we know that $G/\ker(\chi)$ has at least $\phi (|Z(\chi)/\ker(\chi)|)$ faithful irreducible characters.  As we saw in Lemma \ref{p-1}, we have $\phi (Z(\chi)/\ker(\chi)) \ge p -1$.  Since $G/\ker (\chi)$ has at most $p-1$ faithful irreducible characters, we deduce that $\phi (|Z(\chi)/\ker(\chi)|) \le p-1$, and thus, $\phi (|Z(\chi)/\ker(\chi)|) = p-1$.  It is well-known that $\phi (|Z(\chi)/\ker(\chi)|) = p-1$ if and only if $|Z(\chi)/\ker(\chi)| = p$.  We now have that $Z (\chi)/\ker(\chi)$ has $p-1$ nonprincipal irreducible characters.  Since $G/\ker(\chi)$ has at most $p-1$ faithful irreducible characters, we conclude that $\gamma^G$ has a unique irreducible constituent for each character $\gamma \in \irr {Z(\chi)/\ker(\chi)}$.  Observe that $\chi$ is a constituent of $\gamma^G$ for such a character $\gamma$, and it is not difficult to see that this implies that $\chi$ vanishes on $G \setminus Z(\chi)$.  Since $\chi$ was arbitrary, this implies that $G$ is a GVZ-group and $|Z(\chi)/\ker(\chi)| = p$ for all $1_G \ne \chi \in \irr G$.
	
Now, suppose that $G$ is a GVZ-group and $|Z(\chi)/\ker(\chi)| = p$ for all $1_G \ne \chi \in \irr G$.  Consider a character $\chi \in \irr G$.  We see that all the nonzero values of $\chi$ are on elements of $Z (\chi)$ and since $|Z(\chi)/\ker(\chi)| = p$, it follows that all the values of $\chi$ lie in $\mathbb{Q}_p$.
	
Finally, suppose that $G$ is a GVZ-group and every character in $\irr G$ has values in $\mathbb{Q}_p$.  Let $N$ be a normal subgroup of $G$.  If no irreducible character of $G$ has $N$ as a kernel, then the result is true with respect to $N$.  Thus, we may assume that there exists $\chi \in \irr G$ so that $\ker(\chi) = N$.  Observe that $Z(\chi)/N = Z(G/N)$ is cyclic.  Since $\chi$ has values in $\mathbb{Q}_p$, it follows that $Z(\chi)/N$ must have exponent $p$, and so, $Z(\chi)/N$ has order $p$.  Now, every irreducible character in $\irr G$ that has $N$ as its kernel will have a unique character in $\irr {Z(\chi)/N}$ as an irreducible constituent.  Since $G$ is a GVZ-group, we see that each irreducible character in $Z(\chi)/N$ has a unique irreducible constituent upon induction to $G$.  Because $Z(\chi)/N$ has $p-1$ irreducible characters, this shows that there exists $p-1$ irreducible characters of $G$ that have $N$ as their kernel and this proves the result.
\end{proof}

It is easy to find examples to see that if $m > 1$, then the direct product of two ${\rm CM}_n$-groups need not be a ${\rm CM}_n$-group.  However, when $n = p-1$ and the groups are $p$-groups, then the story is different.

\begin{cor}
Let $p$ be a prime and suppose that $H$ and $K$ are ${\rm CM}_{p-1}$-groups that are $p$-groups.  Then $H \times K$ is a ${\rm CM}_{p-1}$-group.
\end{cor} 

\begin{proof}
Observe since $H$ and $K$ are GVZ-groups that $H \times K$ is a GVZ-group.  Also, since all characters in $\irr H$ and $\irr K$ have values in $\mathbb{Q}_p$, it follows that all characters in $\irr {H \times K}$ have values in $\mathbb{Q}_p$.  Applying Lemma \ref{CMp-1}, we see that $H \times K$ is a $\mathrm{CM}_{p-1}$-group. 
\end{proof}

We say $G$ a ${\rm CM}$-group if $G$ is a ${\rm CM}_1$-group. (We note that \cite{saeidi} uses a somewhat different definition for ${\rm CM}_1$-groups.)  Following the usual convention in the literature, we say $G$ is a {\it rational} group if all of the irreducible characters of $G$ are rational.  It is not difficult to see that this is equivalent to the condition that $g$ is conjugate to $g^r$ for every integer $r$ satisfying $(r, \norm{G}) = 1$. We note on page 251 of \cite{BZ} it is mentioned that all ${\rm CM}$-groups are rational.  Also, in Lemma 1.2 of \cite{saeidi}, it is proved that ${\rm CM}$-groups that are $2$-groups are rational. Several of the other results \cite{saeidi} also suggest that $G$ will be a GVZ-group, which turns out to be the case.

\begin{cor}\label{CM}
Let $G$ be a $2$-group.  Then $G$ is a ${\rm CM}$-group if and only if $G$ is a GVZ-group and rational group.
\end{cor}

Appealing to Theorem~\ref{flat = gvz}, Corollary~\ref{CM} can be considered a group-theoretic characterization of those $2$-groups that are ${\rm CM}$-groups. In a private communication, Professor Mann has indicated that Zhmud has conjectured that ${\rm CM}$-groups are metabelian.  This is consistent with Problem 1 of \cite{YBGPPOV1}, which is credited to Zhmud. Furthermore, in \cite{zhmud}, it is shown that ${\rm CM}$-groups are exactly the groups where any two elements with the same normal closure are conjugate. Problem 12.15 of the Kourovka Notebook \cite{Kourovka} asks if such groups are necessarily metabelian.  In light of Corollary~\ref{CM}, we see that showing that these groups are metabelian would provide some evidence that GVZ-groups have bounded derived length (or perhaps are metabelian).  We would like to thank Professor Abdollahi for pointing out the connection between these groups and the problem in the Kourovka Notebook.  

We close by remarking that in \cite{saeidi} it is proved that these groups are metabelian under the additional strong hypothesis that the group has at most five distinct irreducible character degrees.  We note that there is a misstatement in the review of \cite{saeidi} in Math reviews where is stated that the author had proved all ${\rm CM}$-groups are metabelian (see \cite{rev}). 


\end{document}